\documentclass{article}
\usepackage{graphicx} 
\usepackage{amssymb}
\usepackage{amsthm}
\usepackage{amsmath}
\usepackage{latexsym}
\usepackage{amscd}
\usepackage[table]{xcolor}

\usepackage{tikz}

\newtheorem{theorem}{Theorem}
\newtheorem{lemma}[theorem]{Lemma}
\newtheorem{cor}[theorem]{Corollary}

\theoremstyle{definition}
\newtheorem{definition}[theorem]{Definition}

\theoremstyle{remark}

\title{Harmonious sequences in groups with a unique involution }

\date{}

               \author{
  Mohammad Javaheri ~~~~~~~~~~~~~~~~~~~~~~ Lydia de Wolf\\
  \texttt{mjavaheri@siena.edu}~~~~~~~~~~~~~\texttt{ldewolf@siena.edu}\\
  515 Loudon Road \\
  Siena College, School of Science\\
  Loudonville, NY 12211\\
  }

\begin{document}

\maketitle

\begin{abstract}
 We study several combinatorial properties of finite groups that are related to the notions of sequenceability, R-sequenceability, and harmonious sequences. In particular, we show that in every abelian group $G$ with a unique involution $\imath_G$ there exists a permutation $g_0,\ldots, g_{m}$ of elements of $G \backslash \{\imath_G\}$ such that the consecutive sums $g_0+g_1, g_1+g_2,\ldots, g_{m}+g_0$ also form a permutation of elements of $G\backslash \{\imath_G\}$. We also show that in every abelian group of order at least 4 there exists a sequence containing each non-identity element of $G$ exactly twice such that the consecutive sums also contain each non-identity element of $G$ twice. We apply several results to the existence of transversals in Latin squares. 
\end{abstract}


\section{Introduction}

Given a sequence ${\bf g}: g_0,g_1,\ldots, g_m$ in a finite group $G$, the sequence of consecutive quotients ${\bf \bar g}: \bar{g}_0, \bar{g}_1,\ldots, \bar{g}_m$ is defined by letting $\bar{g}_0=g_0$ and $\bar g_i=g_{i-1}^{-1}g_i$ for all $1 \leq i \leq m$. A group is called {\it sequenceable} if there exists a sequence $\bf g$ in $G$ with $g_0=1_G$, where $1_G$ is the identity element of $G$, such that $\bf g$ and $\bf \bar g$ are both permutations of elements of $G$. Gordon \cite{G} proved that a finite abelian group is sequenceable if and only if it has a unique involution (i.e. an element of order 2). 

It is conjectured by Keedwell \cite{Keedwell} that, except for the dihedral groups $D_6, D_8$, and the quaternion group $Q_8$, every non-abelian group is sequenceable. Keedwell's conjecture is supported by the result that all solvable groups with a unique involution, except $Q_8$, are sequenceable \cite{AI1}. See \cite{Ollis} for a survey of results regarding sequenceable groups. 

Considering consecutive products instead of quotients, given a sequence ${\bf g}: g_0,g_1,\ldots, g_m$ in $G$, let the sequence ${\bf \hat g}: \hat{g}_0, \hat{g}_1,\ldots, \hat{g}_m$ be defined by letting  $\hat{g}_0=g_mg_0$ and $\hat{g}_i=g_{i-1}g_i$ for $1\leq i \leq m$. The following definition is consistent with that of Beals et al.~introduced in 1991 \cite{JG}.

\begin{definition}
Let $G$ be a group and $A$ be a finite subset of $G$. A sequence ${\bf g}$ in $A$ is called a {\it harmonious} sequence in $A$ if both $\bf g$ and $\bf \hat g$ are permutations of elements of $A$, in which case we say $A$ is {\it harmonious}.
\end{definition}

All odd groups are harmonious. All abelian groups except the elementary 2-groups and groups with a unique involution are harmonious. A dihedral group of order $n$ is harmonious if and only if $n>4$ and $n=0 \pmod 4$. It is also known that the dicyclic group of order $4n$ is not harmonious if $n$ is odd, and it is harmonious if $n=0 \pmod 4$ or $n=0 \pmod 6$ \cite{Wang}. Moreover, $G^\sharp=G \backslash \{0_G\}$ is harmonious for an abelian group $G$ of order greater than 3 if and only if $G$ does not have a unique involution \cite{JG}, where here and throughout, $0_G$ denotes the identity element of an abelian group $G$. 

Harmonious sequences are special examples of complete mappings. For $G$ a group and $A$ a finite subset of $G$, recall that a \textit{complete mapping} of $A$ is a bijection $\pi: A \rightarrow A$ such that $\{g*\pi(g): g\in A\}=A$. Hall and Paige studied complete mappings in relation to orthogonality problems in Latin squares \cite{HP}.

A harmonious sequence ${\bf g}: g_0,\ldots, g_{m}$ on a subset $A$ of a group $G$ gives rise to the cyclic complete mapping $\pi(g_i)=g_{i+1}$, $0\leq i \leq m$, and conversely, a cyclic complete mapping $\pi$ on $A$ gives rise to the harmonious sequence $g, \pi(g), \pi^2(g), \ldots, \pi^{m}(g)$ in $A$, where $g\in A$. 

The study of complete mappings itself is motivated by the open problem in the study of Latin squares known as the Ryser-Brualdi-Stein conjecture. Recall that a {\it Latin square} is an $n\times n$ array on $n$ symbols such that each symbol appears exactly once in each row and exactly once in each column. A {\it partial transversal} of a Latin square is a collection of cells which do not share any row, column, or symbol. A {\it full transversal} of an $n\times n$ Latin square is a partial transversal with $n$ cells, while a {\it near transversal} is a partial transversal with $n-1$ cells. The Ryser-Brualdi-Stein conjecture states that every odd Latin square has a full transversal and every even Latin square has a near transversal.

Given a group $G$, let $L(G)$ be its multiplication table (which is a Latin square), where the cell $(g,h)$ contains the term $gh$ for $g,h \in G$. Every full transversal of $L(G)$ corresponds to a complete mapping of the group, where each cell $(g,h)$ of the transversal indicates that one maps $g$ to $h$. Paige showed in 1947 that $L(G)$ has a full transversal for abelian group $G$ if and only if $G$ does not have a unique involution \cite{P}. Hall and Paige conjectured in 1955 that for a general group $G$, the Latin square $L(G)$ has a full transversal if and only if the 2-Sylow subgroups of $G$ are either trivial or non-cyclic, and they verified the conjecture in the solvable case \cite{HP}. The Hall-Paige conjecture was eventually proved in 2009 in the general case \cite{E,W}. 

The Ryser-Brualdi-Stein conjecture holds for group-based Latin squares \cite{JG,GH}. Since every odd group admits a harmonious sequence, a stronger result holds in the odd case, namely every group-based odd Latin square has a full transversal satisfying the following additional condition. 

\begin{definition}
    We say a collection of cells in a Latin square is {\it cyclic} if there is an ordering $c_1,\ldots, c_k$ of the cells in such a way that the row number of $c_{i+1}$ matches the column number of $c_i$ for all $1\leq i\leq k$ (where $c_{k+1}=c_1$). 
\end{definition}

If $L(G)$ has a full transversal, then it decomposes into full transversals \cite{VW}; we show in Section \ref{oddgroups} that if $G$ is an odd abelian group of order $n$, then $L(G)$ has $\phi(n)$ disjoint cyclic full transversals (Theorem \ref{abelianpseq}).

An odd Latin square does not necessarily admit a cyclic full transversal (see Figure 1.1). Let $G$ be a group with elements $g_1,\ldots, g_n$. Then the Latin square $\bar L(G)=\left ( g_i^{-1}g_j \right )_{1 \leq i,j \leq n}$ does not admit a cyclic full transversal, since $1_G$ only appears on the main diagonal and no cyclic full transversal intersects the main diagonal if $n>1$. Likewise, an even Latin square does not necessarily admit a cyclic near transversal (see Figure 1.1). If $G$ is a group that is not R-sequenceable (defined below), then $\bar L(G)$ does not admit a cyclic near transversal. We will show in Theorem \ref{evenneartrans} that $L(G)$ has a cyclic near transversal for all even abelian groups $G$.

\begin{figure}[!htb] \label{fig1}
   
    \begin{minipage}{.45\linewidth}

      \centering
        \begin{tabular}{|c|c|c|c|c|}
\hline
\( 0 \)&\( 1\)&\(2\)&\(3\)&\(4\)\\
\hline
\(4\)&\(0\)&\(1\)&\(2\)&\(3\)\\
\hline
\(3\)&\(4\)&\( 0\)&\(1\)&\(2\)\\
\hline
\(2\)&\(3\)&\(4\)&\(0\)&\(1\)\\
\hline
\(1\)&\(2\)&\(3\)&\(4\)&\(0\)\\
\hline
\end{tabular}
    \end{minipage}%
    \begin{minipage}{.55\linewidth}
      \centering
        
        \begin{tabular}{|c|c|c|c|c|c|}
\hline
\( 0 \)&\( 1\)&\(2\)&\(3\)&\(4\)&\(5\)\\
\hline
\(5\)&\(0\)&\(1\)&\(2\)&\(3\)&\(4\)\\
\hline
\(4\)&\(5\)&\( 0\)&\(1\)&\(2\)&\(3\)\\
\hline
\(3\)&\(4\)&\(5\)&\(0\)&\(1\)&\(2\)\\
\hline
\(2\)&\(3\)&\( 4\)&\(5\)&\(0\)&\(1\)\\
\hline
\(1\)&\(2\)&\(3\)&\(4\)&\(5\)&\(0\)\\
\hline
\end{tabular}
    \end{minipage} 
     \caption{The Latin square $\bar L(\mathbb{Z}_5)$ on the left does not admit a cyclic full transversal; the Latin square $\bar L(\mathbb{Z}_6)$ on the right does not admit a cyclic near transversal. }
\end{figure}

Among other combinatorial properties related to sequenceability are R-sequenceability and D-sequenceability. A permutation of non-identity elements $g_0,\ldots, g_{m}$ in a group $G$ is called an \textit{R-sequence} if the consecutive quotients $g_0^{-1}g_1$, $ g_1^{-1}g_2,\ldots,$ $g_{m}^{-1}g_0$ also form a permutation of non-identity elements of $G$. Every abelian group is either sequenceable or R-sequenceable \cite{AKP}. A group is called \textit{D-sequenceable}, if there exists a sequence $\bf g$ in $G$ such that every element of $G$ appears exactly twice in each of $\bf g$ and $\bf \bar g$. All abelian groups are D-sequenceable, and it is conjectured that all finite groups are D-sequenceable \cite{J}. We introduce the following harmonious counterpart of D-sequenceability.

\begin{definition}\label{def1}
Let $G$ be a group and $A$ be a finite subset of $G$. We say ${\bf g}$ is a {\it doubly harmonious} sequence in $A$ if every element of $A$ appears exactly twice in each of ${\bf g}$ and ${\bf \hat g}$, in which case we say $A$ is {\it doubly harmonious}.   
\end{definition}

Every doubly harmonious sequence in $G$ gives rise to a cyclic duplex in $L(G)$. A \textit{duplex} in an $n\times n$ Latin square is a set of $2n$ cells that contains two cells from each row, column, and symbol. Rodney conjectured that every Latin square contains a duplex \cite{R}. Rodney's conjecture has been verified for $L(G)$ if $G$ is solvable \cite{VW}. 

 In Section \ref{dpgrps}, we show that if $G$ is an odd group (possibly trivial) and $H$ is abelian, then $G \times H$ is doubly harmonious (Theorem \ref{dhom}), and so $L(G\times H)$ admits a cyclic duplex. In addition, we show that if $G$ is an abelian group of order at least 4, then $G^\sharp = G \backslash \{0_G\}$ is doubly harmonious (Theorem \ref{abel-drp}). We conjecture that every finite group $G$ is doubly harmonious, and if $G$ is of order at least 4, then $G^\sharp$ is doubly harmonious.

In Section \ref{cpgrps}, we consider two combinatorial properties of groups with a unique involution. We prove that if $G$ is an abelian group with a unique involution $\imath_G$, then $G^\flat=G \backslash \{\imath_G\} $ is harmonious (Theorem \ref{order2CP}), and if $G$ is an abelian group of order at least 10, then $G^\natural=G \backslash \{\imath_G, 0_G\} $ is harmonious (Theorem \ref{order2RCP}).


\section{Strongly harmonious sequences}\label{oddgroups}

We first establish a stronger statement regarding harmonious sequences in odd abelian groups.

\begin{definition}\label{defstrongp}
    We say a sequence $g_0,\ldots, g_{\ell-1}$ in a group $G$ of order $\ell$ is {\it strongly harmonious} if, with indices computed modulo $\ell$, we have
    \begin{itemize}
        \item[i)] for every integer $k$, the terms $g_{i}g_{i+k}$, $0\leq i \leq \ell-1$, form a permutation of elements of $G$.
        \item[ii)] for every integer $0\leq k \leq \ell - 1$, $g_kg_{\ell-k}=1_G$. 
        \end{itemize}
\end{definition}

It follows from (i) with $k=0$ that $g_i^2$, $0\leq i \leq \ell-1$, form a permutation of $G$, hence $G$ must be odd. It then follows from (ii) that $g_0^2=1_G$, hence $g_0=1_G$. 

\begin{theorem}\label{abelianpseq}
 Every odd abelian group is strongly harmonious.  
\end{theorem}

\begin{proof}
    If $G \cong \mathbb{Z}_m$, where $m$ is odd, then one simply lets $g_i=i$, $0\leq i \leq m-1$. It is then sufficient to prove that if an odd abelian group $H$ has a strongly harmonious sequence, then $H\times \mathbb{Z}_m$ has a strongly harmonious sequence, where $m$ is odd. Let $h_0,h_1,\ldots, h_{n-1}$ be a strongly harmonious sequence in $H$. We denote an element of $(h,i) \in G=H \times \mathbb{Z}_m$ by $\genfrac{}{}{0pt}{1}{i}{h}$. Then the sequence below is a strongly harmonious sequence in $G$:
\begin{equation}\label{strongp} {\bf p}:~
\overbrace{\underbrace{\begin{matrix}0 \\ h_0\end{matrix}, \begin{matrix}1 \\ h_1\end{matrix}, \ldots, \begin{matrix}0 \\ h_{n-1}\end{matrix}}_\text{first copy}}^\text{alternating 0 and 1}, \overbrace{\underbrace{ \begin{matrix}1 \\ h_0\end{matrix}, \begin{matrix}2 \\ h_1\end{matrix}, \ldots,\begin{matrix}1 \\ h_{n-1}\end{matrix}}_\text{second copy}}^\text{alternating 1 and 2}, \ldots,\overbrace{\underbrace{\begin{matrix}m-1 \\ h_0\end{matrix},\begin{matrix}0 \\ h_1\end{matrix}, \ldots, \begin{matrix}m-1 \\ h_{n-1}\end{matrix}}_\text{$m$'th copy}}^\text{alternating $m-1$ and 0}.
\end{equation}
To be more precise, for $0\leq i \leq mn-1$, we let
\[
p_i= \begin{cases} \left ( h_{i}, \left \lfloor {i}/{n} 
 \right \rfloor  \right ) & \mbox{if $i+\lfloor i/n \rfloor $ is even;} \\
 \left ( h_i, \left \lfloor {i}/{n} 
 \right \rfloor +1 \right  ) & \mbox{if $i+\lfloor i/n \rfloor $ is odd;}
 \end{cases}
 \]
 where the index of $h_i$ is computed modulo $n$. 
 
 It follows from this definition that $p_{i+jn}=p_i+(0,j)$ for all integers $i,j$. It is clear from \eqref{strongp} that each element of $G$ appears exactly once in ${\bf p}$. To see that each element of $G$ appears exactly once in ${\bf \hat p}$, let $(h,t) \in H \times \mathbb{Z}_m$ be arbitrary. Since ${\bf h}$ is strongly harmonious, there exists $0\leq i \leq n-1$ such that $h=h_i+h_{i+k}$. Let $r\in \mathbb{Z}_m$ such that $p_i+p_{i+k}=(h_i+h_{i+k},r)$. We choose an integer $j$ such that $t=r+2j \pmod m$. Then 
$$p_{i+jn}+p_{i+jn+k}=p_i+(0,j)+p_{i+k}+(0,j)=(h_i+h_{i+k},r+2j)=(h,t).$$
It is straightforward to show that $p_i+p_{mn-i}=0$. 
\end{proof}

Let $G$ be an odd abelian group and $g_0,\ldots, g_{n-1}$ be a strongly harmonious sequence in $G$. Then, for each integer $k$, the cells $(g_i,g_{i+k})$, $0\leq i \leq n-1$, form a full transversal $\Lambda_k$ in $L(G)$, where the indices are computed modulo $n$. Clearly, $\Lambda_k$ is a cyclic full transversal if and only if $\gcd(k,n)=1$.

\begin{cor}
    Let $G$ be an abelian group of odd order $n$. Then $L(G)$ has $n$ disjoint full transversals, $\phi(n)$ of which are cyclic. 
\end{cor}


\section{Doubly harmonious sequences }\label{dpgrps}

We begin this section with an application of results on harmonious groups.

\begin{theorem}\label{abeliandp}
    Every finite abelian group admits a doubly harmonious sequence. 
\end{theorem}

\begin{proof}
    Let $G$ be a finite abelian group. If $G$ is harmonious, then it is doubly harmonious. Thus, suppose that $G$ is not harmonious, and so $G$ is either an elementary 2-group or has a unique involution. If $G$ is an elementary 2-group i.e. $G \cong (\mathbb{Z}_2)^k$, where $k\geq 1$, then $G$ is D-sequenceable \cite{J}, hence it is doubly harmonious (since $a+b=-a+b$ for all $a,b \in(\mathbb{Z}_2)^n$). 
    
    Thus, suppose that $G$ is not an elementary 2-group and has a unique involution. It follows that $G \cong \mathbb{Z}_{2^m} \times H$, where $m > 1$ or $m=1$ and $H$ is nontrivial. Then $\mathbb{Z}_2 \times G$ is not an elementary 2-group and its 2-Sylow subgroup is not cyclic, and so $\mathbb{Z}_2 \times G$ is harmonious. By projecting a harmonious sequence in $\mathbb{Z}_2 \times G$ onto $G$ via the projection $\mathbb{Z}_2 \times G \rightarrow G$, we obtain a doubly harmonious sequence in $G$.  
\end{proof}

We conjecture that all finite groups are doubly harmonious. Clearly, if a finite group $G$ is harmonious, then it is doubly harmonious (since every harmonious sequence can be doubled to obtain a doubly harmonious sequence). As a non-abelian and non-harmonious example, if $k$ is odd, then the dihedral group $D_{2k}$ is not harmonious, but it is doubly harmonious via projecting a harmonious sequence in $D_{4k}$ onto a doubly harmonious sequence in $D_{2k}$. The next theorem (Theorem \ref{dhom}) provides more support for the conjecture.

Recall that a D-sequence in a group $G$ is a sequence ${\bf g}$ in $G$ such that every element of $G$ appears exactly twice in each of ${\bf g}$ and ${\bf \bar g}$ (where ${\bf \bar g}$ is the sequence of consecutive quotients). We say a D-sequence ${\bf g}: g_0,g_1,\ldots, g_{2n-1}$ in a group $G$ is {\it cyclic} if $g_0=g_{2n-1}=0$, where $n=|G|$. It was shown in \cite{J} that every finite abelian group admits a cyclic D-sequence. We use this result in the following lemma. 

\begin{lemma}\label{dseq}
    Let $H$ be an abelian group of order $m$. Then there is a sequence $k_0,\ldots, k_{4m-1}$ in $H$ such that
    \begin{itemize}
    \item[i)] The terms $k_0,k_2,\ldots, k_{4m-2}$ contain each element of $H$ exactly twice. 
    \item[ii)] The terms $k_1,k_3,\ldots, k_{4m-1}$ contain each element of $H$ exactly twice. 
    \item[iii)]  The terms $\hat k_0, \hat k_2, \ldots, \hat k_{4m-2}$ contain each element of $H$ exactly twice. 
    \item[iv)] The terms $\hat k_1, \hat k_3,\ldots, \hat k_{4m-1}$ contain each element of $H$ exactly twice.
    \end{itemize}
\end{lemma}

\begin{proof}
    Let $h_0,\ldots, h_{2m-1}$ be a cyclic D-sequence in $H$ such that $h_0=h_{2m-1}$. We let ${\bf k}$ be the following sequence:
    $${\bf k}: h_0, -h_1, h_2, \cdots, -h_{2m-1},h_1, -h_2, \ldots, h_{2m-1}, -h_0.$$
    To be more precise, we let $h_{2m}=h_0$ and define
    $$k_i=\begin{cases}
        (-1)^ih_i & \mbox{if $0\leq i \leq 2m-1$};\\
        (-1)^{i} h_{i-2m+1} & \mbox{if $2m\leq i \leq 4m-1$}.
    \end{cases}$$
The terms $k_0,k_2,\ldots, k_{4m-2}$ consist of $h_0$, $h_2,\ldots$, $h_{2m-2}$ followed by $h_1$, $h_3,\ldots$, $h_{2m-1}$, which include every element of $H$ exactly twice. Similarly, the terms $k_1,k_3,\ldots, k_{4m-1}$ consist of $-h_1$, $-h_3,\ldots$, $-h_{2m-1}$ followed by $-h_2$, $-h_4,\ldots$, $-h_{2m-2},-h_0$, which include every element of $H$ exactly twice as well. 

We have $\{\hat k_1, \hat k_3,\ldots, \hat k_{4m-1}\} = \{-h_i : 0 \leq i \leq 2m-1 \}$, which contains every element of $H$ exactly twice. Similarly, we have $\{\hat k_0, \hat k_2,\ldots, \hat k_{4m-2}\} = \{h_i : 0 \leq i \leq 2m-1\}$, which contains every element of $H$ exactly twice.
\end{proof}

In the next two lemmas, we show that the direct product of a harmonious group and an abelian group is doubly harmonious.

\begin{lemma}\label{dtodp}
    If $G$ is an even harmonious group and $H$ is abelian, then $G \times H$ is doubly harmonious. 
\end{lemma}

\begin{proof}
   Let ${\bf g}: g_0,\ldots, g_{n-1}$ be a harmonious sequence in $G$ and ${\bf k}:k_0,\ldots, k_{4m-1}$ be the sequence in $H$ obtained in Lemma \ref{dseq}. We define a sequence ${\bf u}: u_0,\ldots$, $u_{2mn-1}$ in $G \times H$ as follows. Given $0\leq i \leq 2mn-1$, we let $i=nj+r$, where $0\leq r \leq n-1$ and $0\leq p \leq 2m-1$. Let
    $$u_i=\begin{cases}
        (g_r, k_{2j}) & \mbox{if $r$ is even;}\\
        (g_r, k_{2j+1}) & \mbox{if $r$ is odd.}
    \end{cases}$$
    
    Hence, 
   \[ {\bf u}:~
\overbrace{\underbrace{\begin{matrix}k_0 \\ g_0\end{matrix}, \begin{matrix}k_1 \\ g_1\end{matrix}, \ldots, \begin{matrix}k_1 \\ g_{n-1}\end{matrix}}_\text{first copy}}^\text{alternating $k_0$ and $k_1$}, \overbrace{\underbrace{ \begin{matrix}k_2 \\ g_0\end{matrix}, \begin{matrix}k_3 \\ g_1\end{matrix}, \ldots, \begin{matrix}k_3 \\ g_{n-1}\end{matrix}}_\text{second copy}}^\text{alternating $k_2$ and $k_3$}, \ldots,\overbrace{\underbrace{\begin{matrix}k_{4m-2} \\ g_0 \end{matrix},\begin{matrix}k_{4m-1}\\ g_1\end{matrix}, \ldots, \begin{matrix}k_{4m-1} \\ g_{n-1}\end{matrix}}_\text{$2m$'th copy}}^\text{alternating $k_{4m-2}$ and $k_{4m-1}$}.
\]

    Given $g \in G$, let $r\in \{0,1,\ldots, n-1\}$ be the unique index such that $g=g_r$. If $r$ is even, the terms $u_{nj+r}=(g_r,k_{2j})$, $0 \leq j \leq 2m-1$, contain every element of $\{g\} \times H$ exactly twice, while if $r$ is odd the terms $u_{nj+r}=(g_r,k_{2j+1})$, $0\leq j \leq 2m-1$, contain every element of $\{g\} \times H$ exactly twice. 

   Next, we show that every element of $G\times H$ appears twice in ${\bf \hat u}$. Given $g\in G$, let $r \in \{0,1,\ldots, n-1\}$ be the unique index such that $g=\hat g_r$. First suppose that $r$ is odd. For $0\leq j \leq 2m-1$, we have
   
   $$\hat{u}_{nj+r}=u_{nj+r-1}u_{nj+r}=(g_{r-1},k_{2j})(g_r,k_{2j+1})=(g, \hat k_{2j+1}).$$
   
   Since every element of $H$ appears twice among the sums $\hat k_{2j+1}$ for $0\leq j\leq 2m-1$, we conclude that the terms $\hat u_{nj+r}$, $0\leq j \leq 2m-1$, contain every element of $\{g\} \times H$ exactly twice. The proof is similar when $r$ is a nonzero even index. If $r=0$, one has
   $$\hat{u}_{nj}=u_{n(j-1)+n-1}u_{nj}=(g_{n-1},k_{2j-1})(g_0,k_{2j})=(g,\hat k_{2j}),$$
    so the terms $\hat u_{nj}$, $0\leq p \leq 2m-1$, contain every element of $\{g\} \times H$ twice as well. Thus ${\bf u}$ is a doubly harmonious sequence in $G \times H$.
\end{proof}

\begin{lemma}\label{extodd2}
If $G$ is an odd group and $H$ is a doubly harmonious group, then $G \times H$ is doubly harmonious. 
\end{lemma}

\begin{proof}
Let ${\bf h}: h_0,\ldots, h_{2n-1}$ be a doubly harmonious sequence in $H$ and ${\bf g}: g_0,\ldots, g_{m-1}$ be a harmonious sequence in $G$. We define a doubly harmonious sequence ${\bf u}:u_0,\ldots, u_{2mn-1}$, in $G \times H$ as follows. Given $0\leq i \leq 2mn-1$, we write $i=2jn+r$, where $0\leq j \leq m-1$ and $0\leq r \leq 2n-1$, and let
\[ u_i= (g_j, h_r). 
\]
That is,
 \[ {\bf u}:~
\underbrace{\begin{matrix}h_0 \\ g_0\end{matrix}, \begin{matrix}h_1 \\ g_0\end{matrix}, \ldots, \begin{matrix}h_{2n-1} \\ g_0\end{matrix}}_\text{all $g_0$}, \underbrace{ \begin{matrix}h_0 \\ g_1\end{matrix}, \begin{matrix}h_1 \\ g_1\end{matrix}, \ldots, \begin{matrix}h_{2n-1} \\ g_1\end{matrix}}_\text{all $g_1$}, \ldots,\underbrace{\begin{matrix}h_0 \\ g_{m-1}\end{matrix},\begin{matrix}h_1 \\ g_{m-1}\end{matrix}, \ldots, \begin{matrix}h_{2n-1} \\ g_{m-1}\end{matrix}}_\text{all $g_{m-1}$}.
\]
It is easily checked that ${\bf u}$ is a doubly harmonious sequence in $G\times H$.
\end{proof}

The following theorem follows from Lemmas \ref{dtodp} and \ref{extodd2}. 
\begin{theorem}\label{dhom}
    If $G$ is harmonious and $H$ is abelian, then $G \times H$ is doubly harmonious.
\end{theorem}

\begin{proof}
    If $G$ is an even group, the claim follows directly from Lemma \ref{dtodp}. If $G$ is odd, the claim follows from Lemma \ref{extodd2}, since every abelian group is doubly harmonious by Theorem \ref{abeliandp}.
\end{proof}

By Theorem \ref{dhom}, if $G$ is an odd group, then $G \times \mathbb{Z}_{2^m}$ is doubly harmonious. These groups provide nontrivial examples of non-harmonious but doubly harmonious groups supporting the conjecture that all finite groups are doubly harmonious. The following result on Latin squares is an immediate consequence of Theorem \ref{dhom}.

\begin{cor}
    If $G$ is an odd group and $H$ is abelian, then $L(G\times H)$ admits a cyclic duplex. 
\end{cor}

In the rest of this section, we show that if $G$ is an abelian group of order at least 4, then $G^\sharp$ is doubly harmonious, where $G^\sharp=G\backslash \{0_G\}$. We say that a doubly harmonious sequence $g_0,g_1,\ldots, g_{m}$ in $G^\sharp$ is {\it special} if $g_0+g_{m-1}=0$ and $g_{m-2}+g_{m-1}=g_m$, or equivalently, if $g_{m-2}=\hat{g}_0$ and $g_m=\hat{g}_{m-1}$. For example, the following sequence is a special doubly harmonious sequence in $\mathbb{Z}^\sharp_{10}$:
$$1, 1, 2, 2, 3, 3, 4, 4, 5, 6,5,7, 6, 8, 7, 9,9,8,$$
and the following sequence is such in $\mathbb{Z}^\sharp_{11}$:
$$1,1,2,2,3,3,4,4,5,5,7,6,6,8,7,9,8,10,10,9.$$
The next lemma shows that these patterns generalize to $\mathbb{Z}_n$ for all integers $n>3$.

\begin{lemma}\label{drplem}
There exists a special doubly harmonious sequence in $\mathbb{Z}_n^\sharp$ for every $n>3$.
\end{lemma}

\begin{proof}
First, let $n=2k$ where $k$ is an integer greater than 1. Let ${\bf s}_1$ and ${\bf s}_2$ be the following sequences:
\begin{eqnarray}\nonumber
    {\bf s}_1:& \underbrace{1,1,2,2, \ldots, k-1,k-1}_{\mbox{pairs $i,i$ for $1\leq i \leq k-1$}},k,k+1;\\ \nonumber 
    {\bf s}_2:& \underbrace{k, k+2, k+1,\ldots, 2k-3,2k-1}_{\mbox{pairs $i,i+2$ for $k\leq i \leq 2k-3$}},2k-1,2k-2.
\end{eqnarray}
Next, let $n=2k-1$ where $k$ is an integer greater than 2. Let ${\bf t}_1$ and ${\bf t}_2$ be the following sequences:
\begin{eqnarray}\nonumber
    {\bf t}_1:& \underbrace{1,1,2,2, \ldots, k-1,k-1}_{\mbox{pairs $i,i$ for $1\leq i \leq k$}},k+1,k;\\ \nonumber 
    {\bf t}_2:&\underbrace{k, k+2, k+1, \ldots, 2k-4,2k-2}_{\mbox{pairs $i,i+2$ for $k\leq i \leq 2k-4$}},2k-2,2k-3.
\end{eqnarray}
It is then straightforward to check that the sequence obtained by joining ${\bf s}_2$ to the end of ${\bf s}_1$ is a special doubly harmonious sequence in $\mathbb{Z}^\sharp_{2k}$, and similarly, the sequence obtained by joining ${\bf t}_2$ to the end of ${\bf t}_1$ is a special doubly harmonious sequence in $\mathbb{Z}^\sharp_{2k-1}$.
\end{proof}

\begin{lemma}\label{drpinduct}
Suppose that $G^\sharp$ admits a special doubly harmonious sequence. Then $\left (G\times \mathbb{Z}_m \right )^\sharp$ admits a special doubly harmonious sequence for every odd integer $m$.
\end{lemma}

\begin{proof}
Let ${\bf g}: g_0,\ldots, g_n$ be a special doubly harmonious sequence in $G^\sharp$. For $1\leq i \leq m-1$, let $\boldsymbol\ell_i$ be the following sequence in $\left ( G\times \mathbb{Z}_m \right )^\sharp$, where $\genfrac{}{}{0pt}{1}{j}{g}$ denotes the element $(g,j) \in G \times \mathbb{Z}_m$, in which $j$ is computed modulo $m$.
\[\boldsymbol\ell_i:~\overbrace{\begin{matrix}i+1 \\ g_0\end{matrix},\begin{matrix}i \\ g_1\end{matrix}, \ldots, \begin{matrix}i+1 \\ g_{n-3}\end{matrix}, \begin{matrix}i \\ g_{n-2}\end{matrix}}^{\mbox{alternating $i+1,i$}},\begin{matrix}i \\ 1_G\end{matrix},\begin{matrix}i \\ 1_G\end{matrix},\begin{matrix}i-2 \\ g_n\end{matrix}, \begin{matrix}i-2 \\ g_{n-1}\end{matrix}.\]
We also let
\[\boldsymbol\ell:~\overbrace{\begin{matrix}1 \\ g_0\end{matrix}, \begin{matrix}0 \\ g_1\end{matrix} \ldots, \begin{matrix}1 \\ g_{n-3}\end{matrix}, \begin{matrix}0 \\ g_{n-2}\end{matrix}}^{\mbox{alternating $1,0$}},\begin{matrix}m-2 \\ g_{n-1}\end{matrix}, \begin{matrix}m-2 \\ g_n\end{matrix}.\]
Then the sequence obtained by joining $\boldsymbol\ell_1,\boldsymbol\ell_2,\ldots,\boldsymbol\ell_{m-1},\boldsymbol\ell$ is a special doubly harmonious sequence in $\left (G\times \mathbb{Z}_m \right )^\sharp$.
\end{proof}

Now, we are ready to prove the main theorem of this section. 

\begin{theorem}\label{abel-drp}
Let $G$ be a finite abelian group of order at least 4. Then $G^\sharp$ is doubly harmonious. 
\end{theorem}

\begin{proof}
  Let $G \cong \mathbb{Z}_{n_1}\times \cdots \times \mathbb{Z}_{n_k}$ be an abelian group, where each $n_i$ is a prime power, $1\leq i \leq k$. If the 2-Sylow subgroup of $G$ is trivial or non-cyclic, then $G^\sharp$ is harmonious (since also $|G|>3$), hence doubly harmonious. Thus, without loss of generality, suppose that $n_1$ is even and $n_i$ is odd for all $2\leq i \leq k$. If $k=1$ then $n_1> 3$ and $G^\sharp \cong \mathbb{Z}_{n_1}^ \sharp$ has a doubly harmonious sequence by Lemma \ref{drplem}. If $k>1$, then $(\mathbb{Z}_{n_1} \times \mathbb{Z}_{n_2})^\sharp=\mathbb{Z}^\sharp_{n_1 n_2}$ has a special doubly harmonious sequence by Lemma \ref{drplem}. It then follows from Lemma \ref{drpinduct} and a simple induction that $G^\sharp$ is doubly harmonious. 
\end{proof}


\section{Groups with a unique involution } \label{cpgrps}

In this section, we show that if $G$ is an abelian group with a unique involution $\imath_G$ then $G^\flat=G\backslash \{\imath_G\}$ is harmonious. We also show that if, moreover, $G$ has order at least 10, then $G^\natural=G \backslash \{\imath_G,0_G\}$ is also harmonious. 

\begin{lemma}\label{CPpowersof2}
    For every positive integer $m$, $\left (\mathbb{Z}_{2^m} \right )^\flat$ is harmonious.
\end{lemma}

\begin{proof}
    Let $k=2^{m-1}$. Then the sequence ${\bf g} : g_0, g_1, \ldots, g_{2k-2}$ given by
    \[ 
  0, k+1, 1, k+2, 2, \ldots,  2k-1, k-1
    \]
    is a harmonious sequence in $\left (\mathbb{Z}_{2^m} \right )^\flat$. More precisely, let
    $$g_i=\begin{cases}
        i/2 & \mbox{if $i$ is even};\\
        k+(i+1)/2 & \mbox{if $i$ is odd}.
    \end{cases}$$
    Modulo $2k$, one has $\hat{g}_i = k+i$ for all $1 \leq i \leq 2k-2$, and $\hat g_0=k-1$. Therefore, ${\bf g}$ is a harmonious sequence in $\left (\mathbb{Z}_{2^m} \right )^\flat$. 
\end{proof}

\begin{theorem}\label{order2CP}
   Let $G$ be an abelian group with a unique element of order 2. Then $G^\flat$ is harmonious. 
\end{theorem}

\begin{proof}
    Let $G \cong \mathbb{Z}_{2^m} \times H$, where $H$ is an odd abelian group of order $n$, and $m\geq 1$. Let ${\bf h} : h_0, h_1, \ldots, h_{n-1}$ be a harmonious sequence in $H$ with $h_0 = 0_H$. 
    
    Case 1: Suppose that $m > 2$. Let $k=2^{m-1}$ and ${\bf g} : g_0, \ldots, g_{2k-2}$ be the harmonious sequence in $\left (\mathbb{Z}_{2^m} \right )^\flat$ given in Lemma \ref{CPpowersof2}. Let ${\bf s} : s_0, \ldots, s_{2k-1}$ be the sequence
    \[
    {\bf s}: 0, k, g_1, \ldots, g_k, g_{k+2}, g_{k+1}, g_{k+4}, g_{k+3}, \ldots, g_{2k-2}, g_{2k-3},
    \]
    where ${\bf s}$ lists $g_1,\ldots, g_{k}$ in the order they appear in ${\bf g}$ but the pairs $g_i,g_{i+1}$ in ${\bf g}$ for $k+1 \leq i \leq 2k-3$, $i$ odd, are reversed in ${\bf s}$. The set of sums $\{ s_{i-1}+s_{i} : 1 \leq i \leq 2k-1\}$ equals $\{ k, \hat{g}_1, ..., \hat{g}_{2k-2} \}$. Let $\boldsymbol\ell_i$, $0\leq i \leq (n-3)/2$, denote the ``block" of $4k$ elements
    \[
    \boldsymbol\ell_i: \overbrace{\begin{matrix}h_{2i} \\ g_0\end{matrix}, \begin{matrix}h_{2i} \\ g_1\end{matrix}, \ldots, \begin{matrix}h_{2i} \\ g_{2k-2}\end{matrix}}^\text{all $h_{2i}$},\overbrace{\begin{matrix}h_{2i+1} \\ s_0\end{matrix}, \begin{matrix}h_{2i+1} \\ s_1\end{matrix} , \ldots, \begin{matrix}h_{2i+1} \\ s_{2k-1}\end{matrix}}^\text{all $h_{2i+1}$},{\begin{matrix} h_{2i+2} \\ k\end{matrix}},
    \]
    and let $\boldsymbol\ell$ be the block of $2k-1$ elements
     $$\boldsymbol\ell: \overbrace{{\begin{matrix} h_{n-1} \\ g_0\end{matrix}}, {\begin{matrix} h_{n-1} \\ g_1\end{matrix}} , \ldots, {\begin{matrix} h_{n-1} \\ g_{2k-2}\end{matrix}}}^\text{all $h_{n-1}$},$$
    in  $ \mathbb{Z}_{2^m} \times H$. Then the sequence
    \[ {\bf p}:~\boldsymbol\ell_0,\boldsymbol\ell_1,\ldots, \boldsymbol\ell_{(n-3)/2}, \boldsymbol\ell,
\]
is a harmonious sequence in $G^\flat$. Clearly, every element of $G^\flat$ appears in ${\bf p}$. To see the same of ${\bf \hat p}$, one checks that the sums consist of the elements $(\hat g_i, 2h_j)$ for $1 \leq i \leq 2k-2$ and $0 \leq j \leq n-1$, the elements $(k, 2h_i)$ for $1 \leq i \leq n-1$, and the elements $(\hat g_0, \hat h_i)$ for $0 \leq i \leq n-1$. Since ${\bf h}$ is harmonious in the odd abelian group $H$ and ${\bf g}$ is harmonious in $(Z_{2^m})^\flat$, these comprise all elements of $G^\flat$.

Case 2: $m = 1$. In this case, Lemma \ref{CPpowersof2} provides ${\bf g} : 0$ and we adjust the definition of ${\bf s}$ to be ${\bf s} :  0, 1$. Then the sequence ${\bf p}$ as constructed above is a harmonious sequence in $G^\flat$.

Case 3: $m = 2$. In this case, Lemma \ref{CPpowersof2} provides ${\bf g} : 0, 1, 3$ and we use ${\bf s}: 0, 2, 3, 1$ in the construction of the harmonious sequence ${\bf p}$.
\end{proof}

The following theorem is a consequence of Theorem \ref{order2CP} and results on harmonious groups. 

\begin{theorem}\label{evenneartrans}
    If $G$ is an even abelian group, then $L(G)$ admits a cyclic near transversal. 
\end{theorem}

\begin{proof} If the 2-Sylow subgroup of $G$ is non-cyclic and $G$ is of order at least 4, then $G^\sharp=G\backslash \{0_G\}$ is harmonious, hence $L(G)$ has a cyclic near transversal. Thus, suppose that the 2-Sylow subgroup of $G$ is cyclic, and so $G \cong \mathbb{Z}_{2^m} \times H$, where $m\geq 1$ and $H$ is an odd abelian group. If $G$ is of order at most 8, then one can prove by inspection that $L(G)$ has a cyclic near transversal (see Figure 4.1). Thus, suppose that $G$ is of order at least 10. By Theorem \ref{order2CP}, $G^\flat=G\backslash \{\imath_G\}$ is harmonious, and so $L(G)$ has a cyclic near transversal in this case as well.
\end{proof}

\begin{figure}[!htb] \label{fig2}
\begin{minipage}{.23\linewidth}
    \centering
        
         \begin{tabular}{|c|c|c|c|}
\hline
0 &  \cellcolor{gray!25}1 &\(2\)&\(3\)\\
\hline
\(1\)&\(2\)&\( \cellcolor{gray!25} 3 \)&\(  0\)\\
\hline
\(\cellcolor{gray!25} 2\)&\(3\)&\( 0\)&\(1\)\\
\hline
\(3\)&\(0\)&\( 1\)&\(2\)\\
\hline
\end{tabular}
    \end{minipage}%
   \begin{minipage}{.35\linewidth}

     \centering
        
         \begin{tabular}{|c|c|c|c|c|c|}
\hline
0 &  1 &\(\cellcolor{gray!25}2\)&\(3\)&\(4\)&\(5\)\\
\hline
\(1\)&\(2\)&\(  3 \)&\(\cellcolor{gray!25}4\)&\(5\)&\( 0\)\\
\hline
\( 2\)&\(3\)&\( 4\)&\(5\)&\(\cellcolor{gray!25} 0\)&\( 1 \)\\
\hline
\(\cellcolor{gray!25} 3\)&\(4\)&\(  5\)&\(0\)&\(1\)&\(2\)\\
\hline
\(4\)&\(\cellcolor{gray!25}5\)&\(0\)&\(1\)&\(2\)&\(3\)\\
\hline
\(5\)&\(0\)&\(1\)&\(  2\)&\(  3\)&\(4\)\\
\hline
\end{tabular}
    \end{minipage}%
    \begin{minipage}{.3\linewidth}
     
     \centering
        \begin{tabular}{|c|c|c|c|c|c|c|c|}
\hline
0 &  1 &\(2\)&\(\cellcolor{gray!25}3\)&\(4\)&\(5\)&\(6\)&\(7\)\\
\hline
\(1\)&\(2\)&\( 3 \)&\(4\)&\( \cellcolor{gray!25} 5\)&\(6\)&\(7\)&\(0\)\\
\hline
\(2\)&\(3\)&\( 4\)&\(5\)&\(6\)&\( \cellcolor{gray!25}7\)&\(0\)&\(1\)\\
\hline
\(3\)&\(4\)&\(5\)&\(6\)&\(7\)&\(0\)&\(\cellcolor{gray!25} 1\)&\( 2\)\\
\hline
\(\cellcolor{gray!25}4\)&\(5\)&\(6\)&\(7\)&\(0\)&\(1\)&\(2\)&\(3\)\\
\hline
\(5\)&\( \cellcolor{gray!25}6\)&\(7\)&\( 0\)&\(1\)&\(2\)&\(3\)&\(4\)\\
\hline
\( 6\)&\(7\)&\(\cellcolor{gray!25} 0\)&\(1\)&\(2\)&\(3\)&\(4\)&\(5\)\\
\hline
\(7\)&\(0\)&\(1\)&\(2\)&\(3\)&\(4\)&\(  5\)&\(6\)\\
\hline
\end{tabular}
    \end{minipage}%
    
     \caption{Examples of cyclic near transversals for $\mathbb{Z}_{4}$, $\mathbb{Z}_6$, and $\mathbb{Z}_8$, respectively. }
\end{figure}

In the rest of this section, we show that if $H$ is an abelian group of odd order and $G \cong  \mathbb{Z}_{2^m} \times H$ is of order at least 10, then $G^\natural=G \backslash \{0_G,\imath_G\}$ is harmonious. We begin with the case $m = 1$.

\begin{lemma}\label{rcp2}
     Let $H$ be an odd abelian group of order at least 5. Then $\left (  \mathbb{Z}_2  \times H \right )^\natural$ is harmonious.  
\end{lemma}

\begin{proof}
    Let $G =  \mathbb{Z}_2 \times H $ and $k = |H|$. Let ${\bf h} : h_1, \ldots, h_{k-1}$ be a harmonious sequence in $H^\sharp=H \backslash \{0_H\}$. Let ${\bf g}: g_1, g_2, \ldots, g_{2k-2}$ be the following sequence in $G$:
\[
    {\bf g}:~{ \begin{matrix} h_1 \\ 0 \end{matrix}},{\begin{matrix} h_1 \\  1\end{matrix}}, {\begin{matrix} h_2 \\ 1 \end{matrix}}, {\begin{matrix} h_2 \\ 0 \end{matrix} }, {\begin{matrix} h_3 \\ 0\end{matrix} }, {\begin{matrix} h_3 \\ 1 \end{matrix}}, \ldots, {\begin{matrix} h_{k-1} \\ 1\end{matrix}}, {\begin{matrix} h_{k-1} \\ 0 \end{matrix}},
\]
where the bottom components alternate between the pairs 0,1 and 1,0. 
 It is clear that every element of $G$, except $(0, 0_H)$ and $(1, 0_H)$, appears once in ${\bf g}$. To see the same for ${\bf \hat g}$, observe that $\hat{g}_{2i} = (1,2h_i)$ for $1\leq i \leq k-1$, yielding the elements $(1,x)$ for all $x \in H^\sharp$. Also, $\hat{g}_{2i+1} = (0,\hat{h}_{i+1})$ for $ 0 \leq i \leq k-2$, yielding the elements $(0,x)$ for all $x \in H^\sharp$ since ${\bf h}$ is harmonious in $H^\sharp$. Thus ${\bf g}$ is a harmonious sequence in $G^\natural$.
\end{proof}

The next three lemmas establish the existence of sequences with certain properties in cyclic abelian groups of even order. We will use these sequences in the construction of harmonious sequences. 

\begin{lemma}\label{ysoddmult4}
    Let $G = \mathbb{Z}_{4k}$, where $k$ is odd and $k \geq 3$. Then there exists a harmonious sequence ${\bf y}: y_1, y_2, \ldots, y_{4k-2}$ in $G^\natural$ such that $y_1 = 2$ and $y_{4k-2} = 2k+1$.
\end{lemma}

\begin{proof}
    Define the sequence ${\bf y} : y_1, \ldots, y_{4k-2}$ as follows:
         \[
    y_i=\begin{cases}
        \frac{i+1}{2}+1 & \mbox{if $1\leq i \leq 2k-3$ and $i$ is odd};\\
        \frac{i}{2}+2k+1 & \mbox{if $1 < i < 2k-3$ and $i$ is even};\\
        3k+2 & \mbox{if $i=2k-2$};\\
        k+2 & \mbox{if $i=2k-1$}; \\
        3k+1 & \mbox{if $i=2k$}; \\
        3k & \mbox{if $i=2k+1$}; \\
        k+1 & \mbox{if $i=2k+2$}; \\
         \frac{i}{2} & \mbox{if $i =0 \pmod 4$ and $ i > 2k+2$}; \\
        \frac{i+1}{2}+2k+2 & \mbox{if $i =1 \pmod 4$ and $ i > 2k+2$}; \\
        \frac{i}{2}+2 & \mbox{if $i =2 \pmod 4$ and $ i > 2k+2$}; \\
        \frac{i+1}{2}+2k & \mbox{if $i =3 \pmod 4$ and $ i > 2k+2$}. \\
    \end{cases}
    \]
    One checks that every element of $\mathbb{Z}_{4k}$, except $0$ and $2k$, appears once in ${\bf y}$. To see the same for ${\bf \hat{y}}$, we first observe that 
    \begin{equation}\label{set1}
        \{{\hat y}_1,{\hat y}_2,\ldots, {\hat y}_{2k-3}\}=\{2k+3, 2k+4,\ldots, 4k-1\}.
    \end{equation}
    In addition, we have
    \begin{equation}\label{set2}
        \{\hat{y}_{2k-2}, \hat{y}_{2k-1}, \hat{y}_{2k}, \hat{y}_{2k+1}, \hat{y}_{2k+2} \}=\{2,4,3,2k+1,1\}.
    \end{equation}
    In the case $k = 3$, one also has $\hat{y}_{2k+3}=5$ and $y_{2k+4} =8.$ It then follows from \eqref{set1} and \eqref{set2} that ${\bf \hat y}$ covers $\mathbb{Z}_{12}^\natural$. When $k>3$, we have 
    $$\{\hat{y}_{2k+3}, \hat{y}_{2k+4}, \hat{y}_{2k+5}, \hat{y}_{2k+6} \}=\{5,8,7,6\},$$ 
    and $\hat y_{2k+4l+j} = \hat y_{2k+j} + 4l$, where $3 \leq j \leq 6$, $1 \leq l \leq \frac{k-3}{2}$. Thus
    \begin{equation}\label{set3}
        \{{\hat y}_{2k+3},{\hat y}_{2k+4},\ldots, {\hat y}_{4k-2}\}=\{5,6,\ldots, 2k-1\} \cup \{2k+2\}.
    \end{equation}
    It then follows from \eqref{set1}, \eqref{set2}, and \eqref{set3} that ${\bf \hat y}$ covers $G^\natural$. 
\end{proof}

\begin{lemma}\label{ysmult8}
     Let $G = \mathbb{Z}_{8k}$, where $k \geq 2$. Then there exists a harmonious sequence ${\bf y}: y_1, y_2, \ldots, y_{8k-2}$ in $G^\natural$ such that $y_1 = 2$ and $y_{8k-2} = 4k+1$.
\end{lemma}

\begin{proof}
    Define the sequence ${\bf y}: y_1, \ldots, y_{8k-2}$ as follows:
         \[
y_i=\begin{cases}
        \frac{i+1}{2}+1 & \mbox{if $1\leq i \leq 4k-3$ and $i$ is odd};\\
        \frac{i}{2}+4k+1 & \mbox{if $1 < i < 4k-3$ and $i$ is even};\\
        6k+2 & \mbox{if $i=4k-2$};\\
        2k+1 & \mbox{if $i=4k-1$}; \\
        6k & \mbox{if $i=4k$}; \\
        6k+1 & \mbox{if $i=4k+1$}; \\
        2k+3 & \mbox{if $i=4k+2$}; \\
        6k+3 & \mbox{if $i=4k+3$}; \\
        2k+2 & \mbox{if $i=4k+4$}; \\
        \frac{i}{2} & \mbox{if $i =0 \pmod 4$ and $ i > 4k+4$}; \\
        \frac{i+1}{2}+4k+2 & \mbox{if $i =1 \pmod 4$ and $ i > 4k+4$}; \\
        \frac{i}{2}+2 & \mbox{if $i =2 \pmod 4$ and $ i > 4k+4$}; \\
        \frac{i+1}{2}+4k & \mbox{if $i =3 \pmod 4$ and $ i > 4k+4$}. \\
    \end{cases}
    \]
    The proof that every element of $\mathbb{Z}_{8k}$, except for $0$ and $4k$, appears exactly once in each of ${\bf y}$ and ${\bf \hat y}$ is similar to the proof of Lemma \ref{ysoddmult4}.
\end{proof}

\begin{lemma}\label{xsmult4}
    Let $G = \mathbb{Z}_{4n}$, where $n \geq 3$. Then there exists a permutation ${\bf x}: x_0, x_1, \ldots, x_{4n-1}$ of elements of $G$ such that
    \begin{itemize}
    \item[i)] $\{ \hat{x}_0+2n\} \cup \{ \hat{x}_1, \hat{x}_2, \ldots, \hat{x}_{4n-1}\} = \mathbb{Z}_{4n}$, \item[ii)] $x_0 + 2n = 2$,
    \item[iii)] $x_{4n-1} = 2n+1$.
    \end{itemize}
\end{lemma}

\begin{proof}
  Define the sequence ${\bf x}$ as follows:
    \[
    x_i=\begin{cases}
        2(i+n)+2 & \mbox{if $0\leq i \leq n$};\\
        i-(-1)^{i+n}n+2 & \mbox{if $n+1\leq i \leq 3n-1$};\\
        2(i+n)+3 & \mbox{if $3n \leq i \leq 4n-1$}. \\
    \end{cases}
    \]
     It is readily checked that every element of $\mathbb{Z}_{4n}$ appears once in ${\bf x}$ and that the second and third properties are satisfied. To see that the first property is satisfied, we observe that, modulo $4n$, the even numbers $\{4i+2 : 1 \leq i \leq n\}$ are given by $\{\hat{x}_i : 1 \leq i \leq n \}$ and the even numbers $\{4i+4 : 0 \leq i \leq n-1\}$ are given by $\{ \hat{x}_i : 3n \leq i \leq 4n-1 \}$. The set of sums $\{ \hat{x}_i : n+1 \leq i \leq 3n-1 \}$ yields the odd numbers in $\mathbb{Z}_{4n}$, except for $2n+3$, which is given by $\hat{x}_0+2n$. 
\end{proof}

We are now ready to prove the theorem. One checks by inspection that if $G$ is an even group of order less than 10 with a unique involution (i.e. $G \cong \mathbb{Z}_{2k}$, $k=1,2,3,4$), then $G^\natural$ is not harmonious. 

\begin{theorem}\label{order2RCP}
     Let $G$ be an abelian group of order at least 10 with a unique involution $\imath_G$. Then $G^\natural=G \backslash \{0_G,\imath_G\}$ is harmonious. 
\end{theorem}

\begin{proof}
    Let $G  \cong \mathbb{Z}_{2^m} \times H$, where $H$ is an odd abelian group, and $|G| \geq 10$.

    Case 1: $m = 1$ and $|H| \geq 5$. The result follows from Lemma \ref{rcp2}.

    Case 2: $m = 2$ and $|H| \geq 3$. We can write $H \cong  \mathbb{Z}_k \times H'$, where $k \geq 3$ is odd, so that $G \cong  \mathbb{Z}_{4k} \times H'$. If $H'$ is the trivial group, the result follows directly from Lemma \ref{ysoddmult4}. Otherwise, let $n=|H'|$ and let ${\bf h}: h_0, \ldots, h_{n-1}$ be a harmonious sequence in $H'$ with $h_0 = 0_{H'}$. Let ${\bf x}: x_0, \ldots, x_{4k-1}$ and ${\bf y}: y_1, \ldots, y_{4k-2}$ be sequences in $\mathbb{Z}_{4k}$ satisfying the properties in Lemmas \ref{xsmult4} and \ref{ysoddmult4}, respectively. Also, define a new sequence ${\bf x'}$ in $\mathbb{Z}_{4k}$ by letting $x'_i = x_i + 2k$, $0\leq i \leq 4k-1$, which also satisfies the properties in Lemma \ref{xsmult4}. Denote by $\genfrac{}{}{0pt}{1}{h_i}{\bf x}$ the ``block" of $4k$ elements in $G$
       \[
   { \begin{matrix} h_i \\ {\bf x}\end{matrix}}: \overbrace{{\begin{matrix} h_i \\ x_0\end{matrix}}, {\begin{matrix} h_i \\ x_1\end{matrix}} , \ldots, {\begin{matrix} h_i \\ x_{4k-1}\end{matrix}}}^{\text{all $h_i$}},
\]
    for $1 \leq i \leq n-1$, and similarly define $\genfrac{}{}{0pt}{1}{h_i}{\bf x'}$, as well as the block of $4k-2$ elements $\genfrac{}{}{0pt}{1}{h_0}{\bf y}$. We claim that the sequence
    \begin{equation} \label{transb} {\bf g}:~
{\begin{matrix} h_0 \\ {\bf y} \end{matrix}}, { \begin{matrix} h_1 \\ {\bf x'}\end{matrix}}, {\begin{matrix} h_2 \\ {\bf x}\end{matrix}}, {\begin{matrix} h_3 \\ {\bf x'} \end{matrix}}, \ldots,  {\begin{matrix} h_{n-2} \\  {\bf x'}\end{matrix}} , {\begin{matrix} h_{n-1} \\ {\bf x} \end{matrix}}
\end{equation}
    is a harmonious sequence in $G^\natural$. Since ${\bf x}$ and ${\bf x'}$ are permutations of $\mathbb{Z}_{4k}$ and ${\bf y}$ is a permutation of $\left ( \mathbb{Z}_{4k} \right )^\natural$, all elements of $G$ except $(0, 0_{H'})$ and $(2k,0_{H'})$ appear once in ${\bf g}$. We will show the same for ${\bf \hat g}$. First, consider elements of the form $(u,0_{H'})$ in $G^\natural$. The first block's sums yield the elements $(\hat{y}_j, 0_{H'})$ for $2 \leq j \leq 4k-2$, leaving out only $(\hat{y}_1,0_{H'}) = (2k+3, 0_{H'})$. Since ${\bf h}$ is a harmonious sequence in $H'$, there exists $1\leq i \leq n-2$ such that $h_i+h_{i+1} = 0_{H'}$. Then, at some transition between blocks in \eqref{transb}, we will have 
    \begin{equation}\nonumber
    \ldots, {\begin{matrix} h_i \\ x_{4k-1}\end{matrix}}, { \begin{matrix} h_{i+1} \\ x'_{0} \end{matrix}}, \ldots \quad \text{ or } \quad \ldots, {\begin{matrix} h_i \\ x'_{4k-1} \end{matrix} }, { \begin{matrix} h_{i+1} \\ x_{0} \end{matrix} }, \ldots 
    \end{equation}
    By the properties of ${\bf x}$ and construction of ${\bf x'}$, we have $x_{4k-1} + x'_0 = x'_{4k-1} + x_0 = 2k+3$. Thus $(u, 0_{H'})$ appears in ${\bf \hat g}$ for all $u\in \mathbb{Z}_{4k}^\natural$. Next, we observe that the consecutive sums in the blocks $\genfrac{}{}{0pt}{1}{h_i}{\bf x} $ and $\genfrac{}{}{0pt}{1}{h_i}{\bf x'}$ in \eqref{transb} provide elements of the form $(\hat{x}_j, 2h_i)$ or $(\hat{x'}_j, 2h_i)$, for all $1 \leq j \leq 4k-1$ and $1 \leq i \leq n-1$. Since $\{\hat{x}_j : 1 \leq j \leq n-1\} = \{\hat{x'}_j : 1 \leq j \leq n-1\}$, by Lemma \ref{xsmult4}, these yield the elements of $G$ of the form $(u,h)$ for all $h\in H'\backslash \{ 0_{H'}\}$ except when $u=\hat{x}_0+2k$. At the transitions between blocks in \eqref{transb}, the sums have second components $h_{i}+ h_{i+1}$ for all $0 \leq i \leq n-1$, and first components $y_{4k-2}+x'_0$, $ x_{4k-1}+x'_0$, or $x'_{4k-1} + x_0$, all of which equal $\hat x_0+2k$. Since ${\bf h}$ is harmonious, these provide the remaining elements of the form $(\hat{x}_0+2k,h)$ for all $h\in H' \backslash \{0_{H'}\}$. This completes the proof that ${\bf g}$ is a harmonious sequence in $\left ( \mathbb{Z}_{4k} \times H' \right )^\natural$.

    Case 3: $m \geq 3$ and $|H| \geq 3$. We write $G = \mathbb{Z}_{8k} \times H'$ for some integer $k$ and odd group $H'$. When $k \geq 3$ and $H'$ is the trivial group, the result follows directly from Lemma \ref{ysmult8}. Otherwise, we construct a harmonious sequence ${\bf g}$ as in Case 2, with two adjustments to the proof. First, let the sequence ${\bf y}: y_1, \ldots, y_{4k-2}$ be as in Lemma \ref{ysmult8}, and the sequence ${\bf x}: x_0, \ldots, x_{8k-1}$ be the sequence obtained in Lemma \ref{xsmult4}. Second, define the sequence ${\bf x'}$ by $x'_i = x_i + 4k$. Using these sequences, the same construction using blocks yields a harmonious sequence for $G^\natural$. 
\end{proof}

Let $B$ be a nonempty subset of $G$ such that there exists a permutation $a_0,\ldots, a_k$ of elements of $G \backslash B$ with $a_0a_1 \cdots a_k=1_G$. Prompted by Theorem \ref{order2RCP}, we conjecture that there exists a function $f$ such that if $|G| \geq f(|B|)$, then $G \backslash B$ is harmonious.


\begin{thebibliography}{20}

\bibitem{AKP} B. Alspach, D.L. Kreher, and A. Pastine, {\it The Friedlander–Gordon–Miller conjecture is true}, {Australas. J. Combin.} {\bf 67} (2017), 11--24.



\bibitem{AI1}B.A. Anderson and E.C. Ihrig, {\it All groups of odd order have starter-translate 2-sequencings}, {\it Australas. J. Combin.} {\bf 6} (1992), 135--146.



\bibitem{JG}  R. Beals, J.A. Gallian, P. Headley, and D. Jungreis, {\it Harmonious groups}. J. Combin. Theory Ser. A {\bf 56} (1991), no. 2, 223--238. 


\bibitem{R}C.J. Colbourn and J.H. Dinitz (eds), {\it The CRC handbook of combinatorial designs}, CRC Press, Boca Raton, FL, 1996.



\bibitem{E}A.B. Evans, {\it The admissibility of sporadic simple groups}, J. Algebra {\bf 321} (2009), no. 1, 105–116.

\bibitem{GH} L. Goddyn and K. Halasz, {\it All group-based latin squares possess near transversals}, J. Combin. Designs {\bf 28} (2020), no. 5, 358--365.


\bibitem{G} B. Gordon, {\it Sequences in groups with distinct partial products}, 
{ Pacific J. Math.} {\bf 11} (1961), no. 4, 1309--1313.

\bibitem{HP}M. Hall and L.J. Paige. {\it Complete mappings of finite groups}, Pacific J. Math. {\bf 5} (1955), 541--549.





\bibitem{J} M. Javaheri, {\it Doubly sequenceable groups}, {\it J. Combin. Designs} {\bf 32} (2024), no. 7, 371--387.

\bibitem{Keedwell} A.D. Keedwell, {\it Sequenceable groups, generalized complete mappings, neofields and block designs} (Reynolds Antoine Casse, editor), Proc. of Tenth Austral. Conf. on Comb. Math. Lecture Notes in Math., vol. 1036, Springer Berlin, Heidelberg, 1983, pp. 49--71.


\bibitem{Ollis} M.A. Ollis, {\it Sequenceable groups and related topics}, Electron. J. Combin. {\bf 20} (2013), no. 2, 1--33. (Version 2).

\bibitem{P} L.J. Paige. {\it A note on finite abelian groups}, Bull. Amer. Math. Soc. {\bf 53} (1947), no. 6, 590--593.



\bibitem{VW}M. Vaughan-Lee and I. M. Wanless, {\it Latin Squares and the Hall–Paige Conjecture}, Bull. London Math. Soc. {\bf 35} (2003), no. 2, 191--195.

\bibitem{Wang}C.-D Wang, {\it On the harmoniousness of dicyclic groups
}, Discrete Math. {\bf 120} (1993), no. 1-3, 221--225. 

\bibitem{W}S. Wilcox, {\it Reduction of the Hall–Paige conjecture to sporadic simple groups}, J. of Algebra {\bf 321} (2009), no. 5, 1407--1428.

\end{thebibliography}
\end{document}